\documentclass[10pt]{article}
\usepackage{amsmath,amssymb,amsthm,amsfonts,amstext,amsbsy,amscd}
\usepackage{color}
\RequirePackage[utf8]{inputenc}
\RequirePackage[T1]{fontenc}
\RequirePackage{lmodern,textcomp,lastpage}
\RequirePackage{graphicx}
\RequirePackage[english]{babel}
\RequirePackage{cancel}

\newtheorem{theorem}{Theorem}
\newtheorem{proposition}{Proposition}
\newtheorem{corollary}{Corollary}
\newtheorem{assumption}{Assumption}

\theoremstyle{definition}
\newtheorem{definition}{Definition}
\newtheorem{remark}{Remark}

\usepackage{amsmath}
\usepackage{amssymb}
\usepackage{graphicx}

\usepackage{color}

\newcommand{\ds}{\displaystyle}

\newcommand{\E}{\mathbb{E}}
\renewcommand{\P}{\mathbb{P}}
\newcommand{\R}{\mathbb{R}}

\newcommand{\M}{\mathbb{M}}

\newcommand{\cC}{\mathcal{C}}
\newcommand{\cE}{\mathcal{E}}
\newcommand{\cF}{\mathcal{F}}
\newcommand{\cL}{\mathcal{L}}
\newcommand{\cN}{\mathcal{N}}
\newcommand{\cB}{\mathcal{B}}

\newcommand{\pare}[1]{\left ( #1 \right )}
\newcommand{\croc}[1]{\left [ #1 \right ]}

\newcommand\indi[1]{{\mathbb{I}}_{\mathnormal{#1}}}

\def \MR{{\rm MR}}

\def \cD{{\cal D}}

\def\1{\mathbf{1}}

\begin{document}


\title{Stochastic processes associated to multidimensional parabolic transmission problems in divergence form\thanks{Supported
    by the Labex B\'ezout.}} 

\author{%
  Pierre~Etor\'e\footnote{Universit\'e Grenoble-Alpes, LJK.
    \textsf{pierre.etore@univ-grenoble-alpes.fr}}
  \and 
 Miguel~Martinez\footnote{Universit\'e Gustave Eiffel, LAMA
    France. \textsf{miguel.math@u-pem.fr}}}

\maketitle

\begin{abstract}
{In this note we define and study the stochastic process $X$ in link with a parabolic transmission operator $(A,\cD(A))$ in divergence form. The transmission operator involves a diffraction condition along a transmission boundary. To that aim we gather and clarify some results coming from the theory of Dirichlet forms as exposed in \cite{Fukushima-et-al-2011} and \cite{Stroock-1988} for general divergence form operators. We show that $X$ is a semimartingale and that it is solution of a stochastic differential equation involving partial reflections in the co-normal directions along the transmission boundary.}
\end{abstract}




\section{Introduction}
\label{sec:introduction}
In this note we aim at tying - with all the necessary rigor - various theoretical results that come from different approaches concerning the probabilistic study of divergence form operators. We also provide a probabilistic representation of the underlying process when the divergence operator is a transmission operator involving a transmission condition across some smooth interface~:~in this particular case, we show that the process is solution of a stochastic differential equation involving partial reflections in the co-normal directions along the transmission boundary. All of these results are natural but we could not find them in the existing literature and would like to record them in print with a proof as they ground the foundations for the study of probabilistic numerical methods for transmission problems (see e.g. \cite{Bossy-al-2010, Etore-Martinez-2020}).
\vspace{0,4 cm}

Our starting point is $a~:~\R^d \rightarrow {\cal M}_{d}\pare{\R}$ some measurable symmetric matrix valued coefficient satisfying the following ellipticity and boundedness condition $({\bf E-B})$~:~

\begin{assumption} \;{\rm ({\bf E-B})~:}
There exists $\lambda, \Lambda\in (0,\infty)$ such that 
\begin{equation}
\label{eq:ell}
\forall x\in \R^d,\;\;\forall \xi\in\R^d,\quad \lambda|\xi|^2\leq \xi^\ast a(x)\xi \leq \Lambda|\xi|^2.
\end{equation}
\end{assumption}

Let us associate to the coefficient $a$ the following unbounded operator $A:\cD(A)\subset L^2(\R^d)\to L^2(\R^d)$ defined by
$$
\cD(A)=\big\{  u\in H^1(\R^d)\text{ with }\sum_{i,j=1}^dD_i(a_{ij}D_ju)\in L^2(\R^d)  \big\}
$$
and
$$
\forall u\in\cD(A), \quad Au=\sum_{i,j=1}^dD_i(a_{ij}D_ju).$$

There exists a closed symmetric Dirichlet form $({\cal E}, {\cal D}[{\cal E}])$ and its corresponding semigroup $(T_t)$ on $L^2(\R^d)$ that are naturally in link with $(A,\cD(A))$. 

We define rigorously these objects and study their relations in Section \ref{sec-dirichlet1}. Using the spectral resolution of the identity associated to $(A,\cD(A))$, we study the regularity in the 'time variable $t$' of $\cE(T_tf,g)$, $f\in L^2(\R^d), g\in\cD[\cE]$ (Subsection \ref{ssec-contruction}). This permits to establish rigorously in Subsection \ref{ssec-stroock} the connection with the results in \cite{Stroock-1988} that are exposed by D.W. Stroock in the $C_b(\R^d)$ setting (Feller semigroup) and to assert the validity of Aronson's estimates for the transition function of $(T_t)$ (see \cite{Aronson-1967}, \cite{Stroock-1988}, \cite{bass}).

We then aim at providing tractable (from a numerical perspective) stochastic representations for the Hunt process $X$ associated to $({\cal E}, {\cal D}[{\cal E}])$. 

Of course, we are in the ideal setting to apply the stochastic calculus for symmetric Dirichlet forms, and we give the Fukushima decomposition of $X$ {\it via} the Revuz correspondence for additive functionals that is presented in Subsection \ref{Sec:Fuku-decomp}. In this general setting a representation of $X$ may be also provided by the so-called Lyons-Zheng decomposition involving reversed-time martingale increments~:~we give a brief insight of the ideas behind this theoretical decomposition in Subsection \ref{Sec:Lyons-Zheng}. Please note that none of the results presented in Section \ref{Sec:stochastic-process} are new (see \cite{Fukushima-et-al-2011}, \cite{Lyons-Zheng}) and we have tried our best to present the ideas in a coherent and assimilable way for a reader that might not be familiar with the subject. 

Then, gradually moving from broad issues to more specific ones,
we focus in Section \ref{sec-dirichlet2} on the particular case where $(A,\cD(A))$ is a transmission operator across some transmission boundary : we present a Skorokhod representation of the Hunt process $X$ associated to $({\cal E}, {\cal D}[{\cal E}])$ in this case. This result is new and constitutes the main contribution of this note (see Subsection \ref{ssec-skoro}). Finally, we give a special attention to the particular case of a diagonal coefficient matrix $a$ that remains constant on each side of the transmission boundary and compare our result to the one obtained in the pionneering paper \cite{Bossy-al-2010} (see Subsection \ref{sec:bossy-al}). We show that, when reduced to this very specific context, our description essentially matches the stochastic differential equation considered in \cite{Bossy-al-2010}.

\section{Dirichlet form and Markovian semigroup associated to general elliptic divergence form operators}
\label{sec-dirichlet1}

\subsection{Definitions and first properties}
\label{ssec-contruction}

To the coefficient matrix $a$, we may associate a closed symmetric Dirichlet form $({\cal E}, {\cal D}[{\cal E}])$  defined on~$L^{2}(\R^d)$  by
\begin{align*}
\left \{
\begin{array}{lll}
{\cal D}\croc{{\cal E}} &=& H^{1}(\R^d),\\
\\
{\cal E}(u,v) &=&\ds \sum_{i,j=1}^d\int_{\R^d} a_{ij}\,D_ju\,D_iv,\hspace{0.5 cm}u,v\in {\cal D}\croc{{\cal E}}
\end{array}
\right .
\end{align*}
(see \cite{Fukushima-et-al-2011}, p111). This closed symmetric Dirichlet form is the starting point of our construction. 

On the underlying Hilbert space $L^{2}(\R^d)$, we denote within this subsection by $(A, {\cal D}(A))$ the (unique) self-adjoint operator associated to $({\cal E}, {\cal D}\croc{{\cal E}})$ and characterized by
\begin{align*}
\left \{
\begin{array}{ll}
{\cal D}( A)\subset {\cal D}\croc{{\cal E}},\\
{\cal E}(u,v) = -\langle Au,v\rangle_{L^2(\R^d)},\hspace{0.3 cm}u\in {\cal D}(A),\,v\in {\cal D}\croc{{\cal E}}
\end{array}
\right .
\end{align*}
 (\cite{Fukushima-et-al-2011}, Theorem 1.3.1 and Corollary 1.3.1 p.21). 

\vspace{0.5 cm}

We aim at identifying this operator - as expected it will turn out that~$(A,\cD(A))$ is nothing else than the operator defined in the Introduction, therefore the common notation.

By the very definition of $(A, {\cal D}(A))$, we have for any $f\in {\cal D}(A)$ and any $g\in C^\infty_c(\R^d)$
$$
-\langle Af, g \rangle_{L^2(\R^d)}= {\cal E}(f, g) =  \sum_{i,j=1}^d\int_{\R^d} a_{ij}\,D_jf\,D_ig
 = -\Big\langle \sum_{i,j=1}^d D_i(a_{ij}D_jf) , g\Big\rangle_{H^{-1}(\R^d),H^1(\R^d)}
$$
where $\sum_{i,j=1}^d D_i(a_{ij}D_jf)$ is understood  in the distributional sense as an element of $H^{-1}(\R^d)$. 
But as $Af\in L^2(\R^d)$ by the definition of $\cD(A)$ the above equality shows that $\sum_{i,j=1}^d D_i(a_{ij}D_jf)\in L^2(\R^d)$
(for any $f\in  \cD(A)$).

Thus, it is proved that ${\cal D}(A)\subseteq \{f\in H^{1}(\R^d)\text{ with } \sum_{i,j=1}^d D_i(a_{ij}D_jf)\in L^2(\R^d)\}$.

In turn (by the density of $C^\infty_c(\R^d)$ in $L^{2}(\R^d)$) the equality permits to identify for any $f\in {\cal D}(A)$,  
$$
Af= \sum_{i,j=1}^d D_i(a_{ij}D_jf).
$$

Let us now prove the reverse inclusion $\{f\in H^{1}(\R^d)\text{ with } \sum_{i,j=1}^d D_i(a_{ij}D_jf)\in L^2(\R^d)\}\subseteq {\cal D}(A)$.

Let $f\in \{f\in H^{1}(\R^d)\text{ with } \sum_{i,j=1}^d D_i(a_{ij}D_jf) \in L^2(\R^d)\}$. 
By the symmetry of the coefficient matrix~$a$ and  integration by parts, it is not hard to prove that for any $v\in {\cal D}(A)$, 
$$\langle Av,f\rangle_{L^2(\R^d)} = -\cE(v,f)= -\sum_{j,i=1}^d\int_{\R^d}a_{ji}D_ifD_jv= \Big\langle \sum_{j,i}D_j(a_{ji}D_if),v\Big\rangle_{L^2(\R^d)}$$ 
and in particular $f\in {\cal D}(A^\ast) \stackrel{\text{def}}{=}\{g\in L^{2}(\R^d)~|~\exists h_g \in L^{2}(\R^d)\;\;\text{s.t.}\;\langle Av,g\rangle = \langle v,h_g\rangle, \forall \,v\in {\cal D}(A)\}$ (see \cite{pazy}). So that we get the reverse inclusion
$$
\{f\in H^{1}(\R^d)  \text{ with } \sum_{i,j=1}^d D_i(a_{ij}D_jf)\in L^2(\R^d)\} \subseteq {\cal D}(A^\ast) = {\cal D}(A)
$$
where the equality comes from the fact that $(A, {\cal D}(A))$ is  self-adjoint. 
Finally, we have proved
\begin{equation}
{\cal D}(A)=\{f\in H^{1}(\R^d)  \text{ with } \sum_{i,j=1}^d D_i(a_{ij}D_jf) \in L^2(\R^d)\}
\end{equation}
and $(A,\cD(A))$ is fully identified as being the same operator of the Introduction \ref{sec:introduction}.

Note that since $a$ is only assumed to be measurable, $C_c^\infty(\R^d)$ - which is a core for the Dirichlet form $({\cal E}, {\cal D}({\cal E}))$ - is not even a subset of ${\cal D}(A)$.

\vspace{0.5cm}

We now turn to the study of the spectral resolution and the semigroup associated to $(\cE,\cD[\cE])$ and $(A,\cD(A))$. For the sake of conciseness we denote $(\cdot,\cdot)=\langle \cdot,\cdot\rangle_{L^2(\R^d)}$ and $||\cdot||=||\cdot||_{L^2(\R^d)}$ till the end of the section.

Since $(-A, {\cal D}(A))$ is a self-adjoint operator on the Hilbert space $L^{2}(\R^d)$ that is non-negative definite, it admits a spectral resolution of the identity $\{E_\gamma~:~\gamma \in [0, \infty)\}$. For any $\gamma\geq 0$ the operator $E_\gamma:L^2(\R^d)\to L^2(\R^d)$ is a self-adjoint projection operator with $(E_\gamma f,f)\geq 0$, $f\in L^2(\R^d)$, and the $E_\gamma$'s form a spectral family with in particular $E_\mu E_\gamma=E_{\mu\wedge \gamma}$, (see \cite{Fukushima-et-al-2011} p18 for a list of properties). The link with $(-A, {\cal D}(A))$ is through 
$$
(-Af,g)=\int_{[0,\infty)} \gamma d(E_\gamma f,g)\hspace{0,4 cm}\forall f\in {\cal D}(A),\;g\in L^{2}(\R^d)
$$
and ${\cal D}(A) = \left \{f\in L^{2}(\R^d)~:~\int_{[0, \infty)} \gamma^2 d(E_\gamma f,f)<\infty\right \}$
(see \cite{Fukushima-et-al-2011} paragraph 1.3.4 p.18).

Consequently, the family of operators $\{T_t\stackrel{\text{def}}{=} {\rm e}^{tA}~:~t>0\}$ is a strongly continuous semigroup of self-adjoint contractions acting on $L^{2}(\R^d)$ (\cite{Fukushima-et-al-2011} Lemma 1.3.2 p.19) and 
$$
(T_t f,g) = \int_{[0,\infty)} {\rm e}^{-\gamma t} d(E_\gamma f,g)\hspace{0,4 cm}\forall f\in L^{2}(\R^d),\;g\in L^{2}(\R^d).
$$
Note that for any $\gamma\geq 0$, $t>0$, and any functions $f\in L^{2}(\R^d)$ and $g\in L^{2}(\R^d)$, we have the commutation property
\begin{align*}
(T_t E_\gamma f, g) &= (E_\gamma f,T_t g)\\
&= \int_{[0,\infty)}{\rm e}^{-\xi t}d_\xi(E_\gamma f, E_\xi g)=
\int_{[0,\gamma]}{\rm e}^{-\xi t}d_\xi(E_\xi E_\gamma f,  g)+\int_{[\gamma,\infty)}{\rm e}^{-\xi t}d_\xi(E_\xi E_\gamma f,  g)\\
&=\int_{[0,\gamma]}{\rm e}^{-\xi t}d_\xi(E_\xi f, g)=\int_{[0,\infty)}{\rm e}^{-\xi t}d_\xi(E_\xi f,  E_\gamma g) =(T_t f, E_\gamma g)=(E_\gamma T_t f, g).
\end{align*}
Note also that for any $f\in L^{2}(\R^d)$ and any $t>0$,
\begin{align*}
\int_{[0, \infty)} \gamma^2 d(E_\gamma T_t f,T_t f)&
=\int_{[0, \infty)} \gamma^2 d_\gamma\pare{\int_{[0, \infty)} {\rm e}^{-\xi t} d_\xi(E_\gamma E_\xi f, T_t f)}\\
&=\int_{[0, \infty)} \gamma^2 d_\gamma\pare{\int_{[0, \infty)} {\rm e}^{-\xi t} d_\xi\pare{\int_{[0,\infty)}{\rm e}^{-\theta t}d_{\theta}(E_\gamma E_\xi f, E_\theta f)}}\\
&=\int_{[0, \infty)} \gamma^2 {\rm e}^{-2\gamma t} d_\gamma(E_\gamma f, f)\\
&\leq \frac{4}{t^2}{\rm e}^{-2}\int_{[0, \infty)} {\rm e}^{-\gamma t} d_\gamma(E_\gamma f, f) = \frac{4}{t^2}{\rm e}^{-2}(T_t f, f)\leq \frac{4}{t^2}{\rm e}^{-2}||f||^2<+\infty,
\end{align*}
where we have used the spectral family property, the associativity of the Stieltjes integral and the inequality
$\gamma^2 \mathrm{e}^{-\gamma t}\leq 4 \mathrm{e}^{-2}/t^2$.
The above inequality ensures that $T_t f \in {\cal D}(A)$ for any $t>0$.

From the fact that $|\frac{d}{dt}{\rm e}^{-\gamma t}|\leq \gamma$ is integrable w.r.t. $d(E_\gamma h,g)$ whenever $h\in {\cal D}(A)$, we deduce from the commutation property that for any $f,g\in L^{2}(\R^d)$ and for any $s>0$
\begin{align*}
-\frac{d}{dt}(T_t f,T_s g) = \int_{[0,\infty)} \gamma {\rm e}^{-\gamma t} d(E_\gamma f, T_s g)\xrightarrow[t\searrow 0+]{} \int_{[0,\infty)} \gamma d(E_\gamma T_s f, g) =  (-A T_s f, g)
\end{align*}
where the limit exists and is well defined (since we have shown that $T_s f \in {\cal D}(A)$). 

If moreover $g\in {\cal D}\croc{{\cal E}}$ then
\begin{align}
\label{eq:deriv-form}
-\frac{d}{ds}(T_{s} f,g) &= -\frac{d}{dt}(T_{s+t} f,g)|_{t=0+} = -\frac{d}{dt}(T_t f,T_s g)|_{t=0+}\nonumber\\
&= (-A T_s f, g) = {\cal E}(T_sf, g).
\end{align}
And since $-\frac{d}{ds}(T_{s} f,g) = -\frac{d}{ds}(T_s g, f)$ by the symmetry property of $T_s$, we deduce that ${\cal E}(T_sf, g) = {\cal E}(f, T_s g)$ for any $f,g \in {\cal D}\croc{{\cal E}}$. 

Consequently, for any $f\in {\cal D}\croc{{\cal E}}$ and using the ellipticity of the coefficient matrix $a$,
\begin{align*}
\lambda||\nabla T_s f||^2\leq &\,\,{\cal E}(T_s f, T_s f) = {\cal E}(T_{2s}f, f)=(-AT_{2s}f, f)\\
&= \int_{[0,\infty)}\gamma {\rm e}^{-2\gamma s}d(E_\gamma f, f) \leq \frac{{\rm e}^{-1}}{s}\int_{[0,\infty)}{\rm e}^{-\gamma s}d(E_\gamma f, f) = \frac{{\rm e}^{-1}}{s}(T_s f, f)\leq \frac{||f||^{2}}{s},
\end{align*}
from which we deduce the fundamental estimate
\begin{equation}
\label{eq:fund-estimate-T}
||\nabla T_s f|| \leq \frac{||f||}{\sqrt{\lambda \, s}},\,\,\;\;\;\;\forall s>0.
\end{equation}
In turn this estimate implies that for any $f\in L^{2}(\R^d)$, $g\in {\cal D}\croc{{\cal E}}$, the function
$$
s\mapsto {\cal E}(T_s f, g)\;\;\text{is integrable on}\;(0,t],
$$
and from \eqref{eq:deriv-form} and the right continuity of $s\mapsto T_s f$ at time $s=0+$ (one may extend $T_0f = f$ as long as no differentiation of $s\mapsto T_s f$ is implied at $s=0+$ when $f\notin {\cal D}(A)$), we deduce the integrated version of \eqref{eq:deriv-form} namely 

$\forall f\in L^{2}(\R^d),\; \forall g\in {\cal D}\croc{{\cal E}}$,
\begin{equation}
\label{eq-stroock-T}
(T_{t} f,g) - (f,g) = -\int_{0}^t\sum_{i,j=1}^d\pare{a_{ij}D_j T_s f, D_ig}ds = -\int_{0}^t{\cal E}(T_s f, g)ds,\hspace{0,3 cm}t\in (0, \infty).
\end{equation}

\subsection{Link with the results of D.W. Stroock \cite{Stroock-1988}}
\label{ssec-stroock}

In his celebrated article {\it Diffusion semigroups corresponding to uniformly elliptic divergence form operators} D.W. Stroock constructs via a regularization procedure a Feller continuous semigroup 
$\{P_t~:~t>0\}$ associated to $a$ with the properties that (with our notations)
\begin{enumerate}
\item the map $t\in [0,\infty)\mapsto P_t\phi \in H^1(\R^d)$ is a weakly continuous map for each $\phi\in C_c^{\infty}(\R^d)$.
\item $\forall \phi, \psi \in C_c^{\infty}(\R^d)$, 
\begin{equation}
\label{eq-stroock-P}
(P_{t} \phi,\psi) - (\phi,\psi) = -\int_{0}^t\pare{a\nabla P_s \phi, \nabla\psi}ds = -\int_{0}^t{\cal E}(P_s \phi, \psi)ds,\hspace{0,3 cm}t\in (0, \infty).
\end{equation}
(Nota~:~please note that there is a sign error in the original version of \cite{Stroock-1988}).
\end{enumerate}

In fact, $\{P_t~:~t>0\}$ determines a unique strongly continuous semigroup $\{\bar{P}_t~:~t>0\}$ of self-adjoint contractions on $L^{2}(\R^d)$. 

The aim of this subsection is to prove the following equality, which clarifies the relationship between the results obtained in \cite{Stroock-1988} and the those provided by the theory of Dirichlet forms \cite{Fukushima-et-al-2011}.
\begin{proposition}
\begin{equation}
\label{eq:equality-sg}
\{\bar{P}_t~:~t>0\} = \{T_t~:~t>0\}\hspace{0,6cm}\text{on}\;\;L^{2}(\R^d).
\end{equation} 
\end{proposition}

\begin{proof}
The semigroup $\{\bar{P}_t~:~t>0\}$ is strongly continuous on $H^{1}(\R^d)$. Moreover, for each $t>0$, $\bar{P}_t$ maps $L^{2}(\R^d)$ into $H^{1}(\R^d)$ and for each $f\in H^{1}(\R^d) = {\cal D}\croc{{\cal E}}$, we have the fundamental estimate
\begin{equation}
\label{eq:fund-estimate-P}
||\nabla \bar{P}_s f|| \leq \frac{1}{\sqrt{\lambda}}\pare{\frac{||f||}{\sqrt{s}}}\wedge ||\nabla f||,\,\,\;\;\;\;\forall s>0.
\end{equation}

(See \cite{Stroock-1988} Theorem II.3.1. p.341).
 
This estimate implies that for for each $f,g\in {\cal D}\croc{{\cal E}}$ and any $t,s>0$, 
\begin{align*}
|{\cal E}(\bar{P}_t f, g) - {\cal E}(\bar{P}_s f, g)|\leq {\Lambda}||\nabla g||\frac{||\bar{P}_{t\vee s - t\wedge s}f - f||}{\sqrt{\lambda\,(t\wedge s)}}\xrightarrow[s\rightarrow t]{} 0,
\end{align*}
which ensures the continuity of $s\mapsto {\cal E}(\bar{P}_s \phi, \psi)$ for any $\phi, \psi \in C_c^{\infty}(\R^d)$. Since $(\bar{P}_t)_{t>0}$ and $({P}_t)_{t>0}$ co\"incide on $C_c^\infty(\R^d)$, we may differentiate in \eqref{eq-stroock-P} (as long as $t>0$) to find that
\begin{equation}
\label{eq:deriv-form-P}
\frac{d}{dt}(\bar{P}_{t} \phi,\psi) = -{\cal E}(\bar{P}_t \phi, \psi),\hspace{0,3 cm}t\in (0, \infty).
\end{equation}
This has to be compared to \eqref{eq:deriv-form}.

Let us now justify rigorously that for any $t>0$, $s\in (0,t)$ and $\phi,\psi\in C_c^{\infty}(\R^d)$, 
\begin{equation}
\label{eq:derivation-produit}
\frac{d}{ds}(T_{s}\phi, \bar{P}_{t-s}\psi) = \frac{d}{du}(T_{u}\phi, \bar{P}_{t-s}\psi)|_{u=s} - \frac{d}{du}(T_{s}\phi, \bar{P}_{t-u}\psi)|_{u=s}.
\end{equation}
We have for sufficiently small $0\leq h<t-s$ and using the strong continuity of $(\overline{P}_t)_{t>0}$
\begin{align*}
|\pare{{T}_{s+h}\phi - T_s\phi, \bar{P}_{t-s+h}\psi - \bar{P}_{t-s}\psi}|&\leq ||\bar{P}_{t-s+h}\psi - \bar{P}_{t-s}\psi||\,||{T}_{s+h}\phi - T_s\phi||\\
&\leq \varepsilon_\psi(h)\pare{\int_{[0,\infty)}{\rm e}^{-2\gamma s}({\rm e}^{-\gamma h} - 1)^2 d(E_\gamma \phi, \phi)}^{1/2}\\
&\leq \varepsilon_\psi(h)\pare{\int_{[0,\infty)}{\rm e}^{-2\gamma s}(\gamma h)^2 d(E_\gamma \phi, \phi)}^{1/2}\\
&\leq h\varepsilon_\psi(h)\pare{\int_{[0,\infty)}{\rm e}^{-\gamma s}\pare{\gamma^2{\rm e}^{-\gamma s}} d(E_\gamma \phi, \phi)}^{1/2}\\
&\leq h\varepsilon_\psi(h)\frac{2{\rm e}^{-1}}{s}||\phi||,
\end{align*}
where as usual $\varepsilon_\psi(\cdot)$ denotes some positive continuous function vanishing at zero.
We deduce that
$$
\frac{1}{h}|\pare{{T}_{s+h}\phi - T_s\phi, \bar{P}_{t-s+h}\psi - \bar{P}_{t-s}\psi}|\xrightarrow[h\rightarrow 0]{}0,
$$
implying \eqref{eq:derivation-produit}.

Hence, from \eqref{eq:derivation-produit} and applying \eqref{eq:deriv-form} and \eqref{eq:deriv-form-P}, we have that
\begin{align}
\label{eq:uniqueness-proof}
\frac{d}{ds}(T_{s}\phi, \bar{P}_{t-s}\psi)=0,\hspace{0,5 cm}s\in (0,t).
\end{align}

Integrating the identity \eqref{eq:uniqueness-proof} on $(0,t)$ and using the time continuity of both semigroups $(T_t)$ and $(\bar{P}_t)$ up to time $s=0+$ gives
\begin{equation}
\label{eq:uniqueness-semigroup}
(T_t\phi, \psi) = (\phi, \bar{P}_t\psi) = (\bar{P}_t\phi, \psi)
\end{equation}
which holds for any $\phi, \psi \in C_c^{\infty}(\R^d)$. Since $C_c^{\infty}(\R^d)$ is dense in $L^2(\R^d)$, using the strong continuity of both semigroups $(T_t)$ and $(\bar{P}_t)$, we finally deduce from \eqref{eq:uniqueness-semigroup} the identification \eqref{eq:equality-sg}.
\end{proof}

Consequently, all results in \cite{Stroock-1988} that are valid for $\{\bar{P}_t~:~t>0\}$ are true for $\{T_t~:~t>0\}$. For example, identifying abusively $\{T_t~:~t>0\}$ with its Feller restriction $\{P_t~:~t>0\}$ on $C^\infty_c(\R^d)$, we deduce that there is a $p\in C\pare{(0,\infty)\times \R^d\times \R^d}$ such that
\begin{equation}
\label{eq:fundamental-solution}
\croc{T_t\phi}(x) = \int_{\R^d}\phi(y)p(t,x,y)dy,\hspace{0,6cm}\ell(dx)-{\rm a.e.},\;\;\;\;\phi\in C_c^\infty(\R^d).
\end{equation}
Moreover, the fundamental function $p$ satisfies the well-known Aronson's estimates for the fundamental solutions of elliptic divergence form operators, namely there exists a constant $M(\lambda, \Lambda, d)\in [1, \infty)$ such that
\begin{equation}
\label{eq:Aronson-estimates}
\frac{1}{Mt^{d/2}}\exp\pare{-M|x-y|^2/t}\leq p(t,x,y)\leq \frac{M}{t^{d/2}}\exp\pare{-|x-y|^2/Mt}.
\end{equation}

Finally, we have the convergence result of \cite{Stroock-1988} (Theorem II.3.1. p.341) that we state roughly without introducing the necessary notations (see \cite{Stroock-1988} for details)~:~if $\{a_n\}_1^\infty\subset{\cal A}(\lambda, \Lambda)$
and $a_n\longrightarrow a$ almost everywhere, then $p^n(t,x,y)\longrightarrow p(t,x,y)$ uniformly on compacts (in $(0,\infty)\times \R^d\times \R^d$) and for each $t\in [0,\infty)$ and $\phi\in C_c^\infty(\R^d)$, $T_t^n\phi\longrightarrow T_t\phi$ in $H^1(\R^ d)$.

\section{Associated stochastic processes}
\label{Sec:stochastic-process}

Since $({\cal E}, {\cal D}\croc{\cal E})$ is a regular Dirichlet form (with the space ${\cal D}\croc{\cal E}\cap C_c(\R^d)$ or $C_c^\infty(\R^d)$ as a special standard core, see e.g. Exercice 1.4.1 in \cite{Fukushima-et-al-2011}), we are in position to apply Theorem 7.2.1 p. 380 of \cite{Fukushima-et-al-2011}. 

We may associate to $({\cal E}, {\cal D}\croc{\cal E})$ and its corresponding semigroup $(T_t)$ a Hunt process, symmetric w.r.t the Lebesgue measure $\ell(dx)$ on $\R^d$. We shall denote by
$\M = \pare{\Omega, ({\cal F}_t)_{t\geq 0}, {\cal F}, (X_t)_{t\geq 0}, (P^x)_{x\in \R^d}}$ this Hunt process, with $X = (X^1, \dots, X^d)$.
The correspondence with $({\cal E}, {\cal D}\croc{\cal E})$ and $(T_t)$ is through
\begin{equation}
\label{eq:Tt-X}
E^x[f(X_t)]=T_tf(x),\quad\forall f\in L^2(\R^d),\;\forall t\geq 0, \forall x\in \R^d
\end{equation}
(see the discussion p160 in \cite{Fukushima-et-al-2011}, at the beginning of Section 4.2).

The aim of this section is to present various representations of $\M$ in various contexts. We start with the general case and then present a Skorokhod decomposition in the specific case where our Dirichlet form is associated to a transmission parabolic operator in divergence form.

\subsection{The Revuz correspondence for additive functionals and the Fukushima decomposition}
\label{Sec:Fuku-decomp}
\subsubsection{A reminder of the Revuz correspondence for additive functionals}
Let us also denote by $\{R_{\alpha}~:~\alpha>0\}$ the Markovian resolvent kernel of the Markovian transition function $\{\hat{p}(t,x,dy):= p(t,x,y)dy~:~t>0\}$.
Then, for any $\alpha >0$, $f\in {\cal B}_b(\R^d)$ and $x\in \R^d$, $R_\alpha f(x) = \int_{\R^d} r_{\alpha}(x,y)f(y) dy$ with $r_{\alpha}(x,y) = \int_0^\infty {\rm e}^{-\alpha t}p(t,x,y)dt$.

Denote by $S$ the set of positive Radon measures on $(\R^d, {\cal B}(\R^d))$. For $\mu\in S$ define $R_{1}\mu(x) = \int_{\R^d} r_1(x,y)\mu(dy)$ ($x\in \R^d$) and introduce the subset of {\it finite energy measures}
\begin{align*}
S_0&:= \left \{\mu \in S~:~\sup_{v\in {\cal D}\croc{{\cal E}}\cap C_c(\R^d)}\int_{\R^d}\frac{|v(x)|}{\;\;||v||_{{\cal E}_1}}\mu(dx)<\infty\right \},
\end{align*}
(where we follow the notations of \cite{Fukushima-et-al-2011}).

Finally, introduce
$$
S_{00}:= \{\mu \in S_0~:~\mu(\R^d)<\infty,\;\;||R_1\mu(.)||_{\infty}<\infty\}.
$$

Let us denote respectively by ${\bf A}^+_{c}$ and ${\bf A}^+_{c,1}$ the families of all Positive Continuous Additive Functionals (PCAF in short) (resp. the family of all PCAF in the strict sense) associated to $\M$ (for the distinction between ${\bf A}^+_{c}$ and ${\bf A}^+_{c,1}$, see \cite{Fukushima-et-al-2011} the introduction of Section 5.1).

The {\it Revuz correspondence} asserts that there is a one-to-one correspondence (up to equivalence of processes) between 
${\bf A}^+_{c}$ and $S$. This correspondence permits to construct for any $\mu\in S_{00}$ a unique PCAF in the strict sense $A\in {\bf A}^+_{c,1}$ such that
\begin{equation}
\label{eq:revuz-correspondence-restricted}
\forall x\in \R^d,\;\;\;\;\;E^x\int_0^\infty {\rm e}^{-t}dA_t = {R_{1}\mu}(x).
\end{equation}
 (see for e.g. Theorem 5.1.4 in \cite{Fukushima-et-al-2011}).
 
In order to get a bijective map, introduce a new subset $S_1$ of $S$ defined by $\mu\in S_1$ if 
there exists a sequence $(E_n)_{n\geq 0}$ of Borel finely open sets increasing to $\R^d$ satisfying that $\indi{E_n}.\mu \in S_{00}$ for each $n$.
Then, there is a one-to-one correspondence between $S_1$ and ${\bf A}^+_{c,1}$ (up to equivalence) which is given by relation \eqref{eq:revuz-correspondence-restricted} whenever $\mu \in S_{00}$. The set of measures $S_1$ is called the set of {\it smooth measures (in the strict sense)}.

Let us introduce ${\cal D}\croc{\cal E}_{b}$ (resp. ${\cal D}\croc{\cal E}_{b,\text{loc}}$) the space of essentially bounded functions belonging to ${\cal D}{\croc{\cal E}}$  (resp. locally to ${\cal D}{\croc{\cal E}}$. A function $u$ is in ${\cal D}\croc{\cal E}_{b,\text{loc}}$ if for any compact set $G$, there exists a bounded function $\omega\in{\cal D}\croc{\cal E}$ such that $u= \omega$,\; $\ell(dx)$-a.e. on $G$). 

For $u\in {\cal D}\croc{\cal E}_{b}$, we may associate a unique positive Radon measure $\mu_{\langle u\rangle}\in S$, satisfying
\begin{equation}
\label{eq:radon-measure-S}
\int_{\R^d}f(x)\mu_{\langle u\rangle}(dx) = 2{\cal E}(uf,u) - {\cal E}(u^2, f),\hspace{0,3 cm}\forall f\in {\cal D}\croc{\cal E}\cap C_c(\R^d).
\end{equation}
Observe that the positivity of the right hand side in \eqref{eq:radon-measure-S} comes from
\begin{align*}
0\leq E^x\croc{(u(X_t) - u(x))^2} &= E^x\croc{u^2(X_t) - 2u(x)u(X_t) + u^{2}(x)}\\
&= (T_tu^2 - u^2)(x) - 2u(x)(T_tu - u)(x).
\end{align*}
Taking the scalar product with $f\in {\cal D}\croc{\cal E}\cap C_c(\R^d)$ and dividing by $t$ gives
\begin{align*}
0&\leq  \pare{\frac{T_tu^2 - u^2}{t},\,f} - {2}\pare{\frac{T_t u - u}{t},\,uf}
\end{align*}
which tends to $2{\cal E}(uf,u) - {\cal E}(u^2, f)$ as $t$ tends to $0+$ whenever $u\in {\cal D}(A)$ ; the positivity in the case $u\in {\cal D}\croc{\cal E}_{b}$ may be obtained by a density argument.

If $u\in {\cal D}\croc{\cal E}_{b,\text{loc}}$, we may construct $\mu_{\langle u\rangle}\in S$ with the help of a sequence $(G_n)_{n\geq 0}$ of relatively compact open sets such that $\overline{G_n}\subset G_{n+1}$ and $\bigcup_{n\geq 0} G_n = \R^d$. Let $(u_n)_{n\geq 0}$ a sequence of functions in ${\cal D}\croc{\cal E}_{b}$ satisfying $u_n = u$ on $G_n$. There is no ambigu\"ity in defining $\mu_{\langle u\rangle} = \mu_{\langle u_n\rangle}$ on $G_n$ because the construction is consistent (since $\mu_{\langle u_n\rangle} = \mu_{\langle u_{n+1}\rangle}$ on $G_n$). For an account on the above assertions, please refer to \cite{Fukushima-et-al-2011} Section 3.2.

\subsubsection{The Fukushima decomposition}
Note that obviously $({\cal E}, {\cal D}\croc{\cal E})$ is strong local, so we may apply Theorem 5.5.5 in \cite{Fukushima-et-al-2011}.

Suppose that a function $u$ satisfies the following conditions~:
\begin{enumerate}
\item $u\in {\cal D}\croc{\cal E}_{b,\text{loc}}$ , $u$ is finely continuous on $\R^d$.
\item $\indi{G}.\mu_{\langle u\rangle}\in S_{00}$ for any relatively compact open set $G$.
\item $\exists \varrho = \varrho^{(1)} - \varrho^{(2)}$ with $\indi{G}.\varrho^{(1)}, \indi{G}.\varrho^{(2)} \in S_{00}$ for any relatively compact open set $G$ and
$$
{\cal E}(u,v) = (\varrho, v),\;\;\;\;\forall v\in C_c^{\infty}(\R^d).
$$
(Note that even though $u$ is not formally in ${\cal D}\croc{\cal E}$, the quantity ${\cal E}(u, v)$ is well-defined because $v$ has compact support and $u\in {\cal D}\croc{\cal E}_{b,\text{loc}}$).
\end{enumerate}
Let $A^{(1)}$,  $A^{(2)}$, and $B$ be PCAF's in the strict sense with Revuz measures $\varrho^{(1)}, \varrho^{(2)}$, and $\mu_{\langle u\rangle}$ respectively.
Then, Theorem 5.5.5 in \cite{Fukushima-et-al-2011} asserts that
\begin{equation}
u(X_t) - u(X_0) = M_t^{[u]} + N_t^{[u]},\hspace{0,3 cm}\P^x-a.s,\; \forall x\in \R^d.
\end{equation}
Here, 
\begin{equation}
N^{[u]} = -A^{(1)} + A^{(2)},\hspace{0,3 cm}\P^x-a.s, \;\forall x\in \R^d
\end{equation}
and $M^{[u]}$ is a local Additive Functional in the strict sense such that for any relatively compact set $G$,
$$
E^x M_{t\wedge \tau_G}^{[u]} = 0,\hspace{0,3 cm}\forall x\in G
$$
and
$$
E^x\croc{(M_{t\wedge \tau_G}^{[u]})^2}= \E^x B_{t\wedge \tau_G},\hspace{0,3 cm}\forall x\in G,
$$
where $\tau_G=\inf(s>0~:~X_s\notin G)$ stands for the first leaving time from $G$ (with the convention $\inf \emptyset = \infty$) and $B$ denotes the PCAF in the strict sense with Revuz measure $\mu_{\langle u\rangle}$.

\subsection{An insight to the Lyons-Zheng decomposition for diffusions associated to divergence form operators}
\label{Sec:Lyons-Zheng}
For the seek of completeness, in this paragraph we present briefly an insight on the Lyons-Zheng decomposition \cite{Lyons-Zheng} for the stochastic process in correspondence with a divergence form operator.

\subsubsection{Time reversal for diffusions associated to divergence form operators}
Assume for a moment that $a\in {\cal A}\pare{\lambda, \Lambda}$ is very smooth and belongs to $C^{\infty}\pare{\R^d \to {\cal M}_{d}\pare{\R}}$.

In this smooth case, the process $\{X_t = (X_t^1, \dots, X_t^d)~:~0\leq t\leq T\}$ in correspondence with $({\cal E}, {\cal D}\croc{\cal E})$ and constructed in the previous section becomes a diffusion process with values in $\R^d$ and well-known results from the general theory for solutions of stochastic differential equations ensure that $(X_t)$ is solution of 
\begin{align}
\label{eq:EDS-Hunt-2}
X_t^k &= x_k + \int_0^t \sum_{j=1}^d \sigma_{kj}(X_s)dW_s^{j} + \int_{0}^t \sum_{j=1}^d \partial_j a_{kj}(X_s)ds\hspace{0,3 cm}t\geq 0,\;{\mathbb P}^x-a.s.,\;\forall x=(x_1,\dots, x_d)\in \R^d,
\end{align}
where $W=(W^1,\dots,W^d)$ denotes a $d$-dimensional standard Brownian motion starting from zero and $\sigma:\R^d\to {\cal M}_d(\R)$ denotes the positive square-root of coefficient $2a$ {\it i.e.} the positive matrix real valued coefficient satisfying
$$
\sigma\sigma^*(x)=2a(x),\quad\forall x\in \R^d.
$$ 

We will denote by ${\cal L}$ the classical generator of $(X_t)$,
\begin{equation}
\label{eq:def-cL}
{\cal L} = \sum_{i,j=1}^d \partial_{i}\pare{a_{ij}\partial_j} = \sum_{i,j=1}^d a_{ij}\partial^2_{ij} + \sum_{i,j=1}^d \pare{\partial_{i}a_{ij}}\partial_{j}
\end{equation}
acting on $C^{2}(\R^d)$ real valued functions. 

Fix $x=(x_1,\dots, x_d)\in \R^d$. For any fixed $0\leq t<1$, denote $\overline{\cal L}^{x}_t$ the second order differential operator
\[
\overline{\cal L}^{x}_t = \sum_{i,j=1}a_{ij}\partial^{2}_{ij} - \sum_{i,j=1}^d \pare{\partial_{i}a_{ij}}\partial_{j}+ \croc{p(1-t,x,.)}^{-1}\sum_{i,j=1}^d \partial_j\pare{\pare{2a_{ij}} p(1-t,x,.)}\partial_{i}
\]
acting also on $C^{2}(\R^d)$ real valued functions and where $p$ denotes the fundamental solution \eqref{eq:fundamental-solution}. (Nota~:~the reader should be careful with the $2$ factor appearing in the last integral, that is due to the fact that we are considering $\sigma$ as the square root of $2a$ and not $a$. Also, since $a$ is assumed to be smooth, classical estimates for $p$ ensure that the term ${p(1-t,x,.)}^{-1}\partial_j\pare{\pare{2a_{ij}} p(1-t,x,.)}\partial_{i}$ is well defined). 

Consider $\{\overline{X}_t := X_{1-t}~:~t\in [0,1)\}$ the time reversed process of $(X_t)$. The fact that the time reversal of a Markov process is again a (weak) Markov process with respect to the reversed filtration traces back to the seminal result of K.L. Chung and J.B. Walsh \cite{Chung-Walsh}. Moreover, it is proved in \cite{Haussmann-Pardoux} that $\pare{\overline{X}_t}_{0\leq t<1}$ is a diffusion process with generator $\overline{\cal L}$ up to time $1$ excluded. The process $\{\overline{X}_t = (X_{1-t}^1, \dots, X_{1-t}^d)~:~0\leq t< 1\}$ is solution of 
\begin{align*}
\overline{X}_t^k = X_1 &+ \int_0^t \sum_{j=1}^d \sigma_{kj}(\overline{X}_s)d\beta_s^{j} - \int_{0}^t \sum_{j=1}^d \partial_j a_{kj}(\overline{X}_s)ds\nonumber\\
 &+ \int_0^t\croc{p(1-t,x,\overline{X}_s)}^{-1}\sum_{i,j=1}^d \partial_j\pare{\pare{2a_{ij}} p(1-s,x,\overline{X}_s)}ds,\;{\mathbb P}^x-a.s.,\;\forall x=(x_1,\dots, x_d)\in \R^d,
\end{align*}
where $\beta=(\beta^1,\dots,\beta^d)$ denotes a $d$-dimensional standard Brownian motion starting from zero and adapted to the filtration $\left \{{\cal F}^{\overline{X}}_t := \sigma\pare{X_{1-u}~:~u\in [0,t]}~\mid~t\in [0,1)\right \}$.

\subsubsection{The Lyons-Zheng decomposition}

We now make the following observation. Fix $0<\varepsilon<1$, then for any $\varphi\in C_{c}^{\infty}(\R^d)$ and arbitrary $t\in [\varepsilon, 1]$
\begin{align*}
\varphi(X_t) - \varphi(x) &= \frac{1}{2}\pare{\varphi(X_t) - \varphi(x)} + \frac{1}{2}\croc{\pare{\varphi(X_{1-(1-t)}) - \varphi(X_{1-(1-\varepsilon)})} + \pare{\varphi(X_{\varepsilon}) - \varphi(x)}}\\
&=\frac{1}{2}\pare{\varphi(X_t) - \varphi(x)} - \frac{1}{2}\pare{\varphi(\overline{X}_{1-\varepsilon})  - \varphi(\overline{X}_{1-t})} + \frac{1}{2}\pare{\varphi(X_{\varepsilon}) - \varphi(x)}
\end{align*}
and applying It\^o's formula, we find that for all $t\in [\varepsilon,1]$ 
\begin{align}
\label{eq:renversement}
&\varphi(X_t) - \varphi(x)\nonumber\\
&= \frac{1}{2}\pare{\int_0^t \pare{\nabla \varphi(X_u), \sigma(X_u)dW_u} + \int_0^t{\cal L}\varphi(X_u)du} - \frac{1}{2}\pare{\int_{1-t}^{1-\varepsilon} \pare{\nabla \varphi(\overline{X}_u), \sigma(\overline{X}_u)d\beta_u} - \int_{1-t}^{1-\varepsilon}\overline{\cal L}^{x}_u\varphi(\overline{X}_u)du}\nonumber\\
&\hspace{0,3 cm} + \frac{1}{2}\pare{\varphi(X_{\varepsilon}) - \varphi(x)}\nonumber\\
&=\frac{1}{2}M_t^{\varphi} + \frac{1}{2}\tilde{N}_t^{\varphi,\varepsilon} + \frac{1}{2}\int_{\varepsilon}^{t}({\cal L} - \overline{\cal L}_{1-u}^{x})\varphi(X_u)du\nonumber + \frac{1}{2}\pare{\varphi(X_{\varepsilon}) - \varphi(x) + \int_0^{\varepsilon}{\cal L}\varphi(X_u)du}\nonumber\\
\end{align}
where we have set $\ds M_t^{\varphi} := \int_0^t \pare{\nabla \varphi(X_u), \sigma(X_u)dW_u}$ and $\ds \tilde{N}_t^{\varphi,\varepsilon} := -\int_{1-t}^{1-\varepsilon} \pare{\nabla \varphi(\overline{X}_u), \sigma(\overline{X}_u)d\beta_u}$. 

We may write 
\begin{align*}
\tilde{N}_t^{\varphi,\varepsilon} &= -\int_{1-t}^{1-\varepsilon} \pare{\nabla \varphi(\overline{X}_u), \sigma(\overline{X}_u)d\beta_u}\\
&=\int_0^{1-t}\pare{\nabla \varphi(\overline{X}_u), \sigma(\overline{X}_u)d\beta_u} -\int_{0}^{1-\varepsilon} \pare{\nabla \varphi(\overline{X}_u), \sigma(\overline{X}_u)d\beta_u}\\
&=N^{\varphi}_{1-t} - N_{1-\varepsilon}^{\varphi},
\end{align*} 
where $\pare{N^{\varphi}_\theta}$ stands for the ${\cal F}^{\overline{X}}$ martingale $\ds\pare{\int_0^{\theta}\pare{\nabla \varphi(\overline{X}_u), \sigma(\overline{X}_u)d\beta_u}}_{\theta\in [0,1)}$.

Now for any $z\in \R^d$
\begin{align*}
\frac{1}{2}\pare{{\cal L} - \overline{\cal L}_{1-u}^{x}}\varphi(z) &= \sum_{i,j=1}^d \pare{\partial_{i}a_{ij}}\partial_{j}\varphi(z) - \croc{p(u,x,z)}^{-1}\sum_{i,j=1}^d \partial_j\pare{a_{ij} p(u,x,z)}\partial_{i}\varphi(z)\\
&={\croc{{\sum_{i,j=1}^d \pare{\partial_{i}a_{ij}}\partial_{j}\varphi(z)} - {\sum_{i,j=1}^d \pare{\partial_j a_{ij}}\partial_{i}\varphi(z)}} }- \croc{p(u,x,z)}^{-1}\sum_{i,j=1}^d a_{ij}\partial_jp(u,x,z)\partial_{i}\varphi(z)\\
&=- \croc{p(u,x,z)}^{-1}\sum_{i,j=1}^d a_{ij}\partial_jp(u,x,z)\partial_{i}\varphi(z)
\end{align*}
where terms cancel due to the symmetry of the coefficient matrix $a$.

Of course $\ds\lim_{\varepsilon \searrow 0+}\pare{\varphi(X_\varepsilon) - \varphi(x) + \int_{0}^{\varepsilon}{\cal L}\varphi(X_s)ds} = 0$.
Applying the martingale convergence theorem ensures that the martingale term ${N}_{1-\varepsilon}^{\varphi}$ tends $\P^x -a.s.$ to $\ds N_1^{\varphi} = \int_0^{1}\pare{\nabla \varphi(\overline{X}_u), \sigma(\overline{X}_u)d\beta_u}$ as $\varepsilon$ tends to zero. 

In turn, this ensures from \eqref{eq:renversement} that the limit
\begin{equation}
\label{eq:limit-dlog-density}
\lim_{\varepsilon \searrow 0+}\int_\varepsilon^t \croc{p(u,x,X_u)}^{-1}\sum_{i,j=1}^d a_{ij}\partial_jp(u,x,X_u)\partial_{i}\varphi(X_u)du 
\end{equation}
exists $\P^x-a.s.$

Coming back to \eqref{eq:renversement} and taking limits as $\varepsilon \searrow 0+$ yields the Lyons-Zheng decomposition of $\pare{X_t}$ when the coefficient matrix $a$ is smooth~:~for all $\varphi\in C_c^{\infty}$, $\forall t\in [0,1), \P^{x}-a.s.$
\begin{align}
\label{eq:renversement-2}
\varphi(X_t) - \varphi(X_s) 
&=\frac{1}{2}M_t^{\varphi} + \frac{1}{2}\tilde{N}^{\varphi}_t - \int_{0}^{t}\croc{p(u,x,X_u)}^{-1}\sum_{i,j=1}^d a_{ij}\partial_jp(u,x,X_u)\partial_{i}\varphi(X_u)du
\end{align}
with $\tilde{N}_t^{\varphi} = N_{1-t}^{\varphi} - N_{1}^{\varphi}$ an increment of a time-reversed martingale. The quadratic variations of the martingales involved in the decomposition are given by $$\ds \langle M^{\varphi}\rangle_t = \int_0^t \sum_{i,j=1}^d 2a_{ij}(X_s)\partial_{i}\varphi(X_s)\partial_j\varphi(X_s)ds,\;\;\text{and}\;\;\ds \langle N^{\varphi}\rangle_t = \int_0^t \sum_{i,j=1}^d 2a_{ij}(\overline{X}_s)\partial_{i}\varphi(\overline{X}_s)\partial_j\varphi(\overline{X}_s)ds.$$
\vspace{0,4 cm}

We make the following observations~:~
\begin{itemize}{}
\item If we have proved \eqref{eq:limit-dlog-density}, it is important to notice however that there is no limit to the deterministic quantity $\ds \int_\varepsilon^t \croc{p(u,x,z)}^{-1}\sum_{i,j=1}^d a_{ij}\partial_jp(u,x,z)\partial_{i}\varphi(z)du$ as $\varepsilon$ tends to $0+$. This can be easily seen by performing the computation from the explicit Laplacian case where $p$ is just the Gaussian transition density of some Browian motion started from $x$.
\item For any $\varepsilon>0$ the random variable $\ds \tilde{N}_t^{\varphi,\varepsilon} = -\int_{1-t}^{1-\varepsilon} \pare{\nabla \varphi(\overline{X}_u), \sigma(\overline{X}_u)d\beta_u}$ is measurable w.r.t the sigma field $$\sigma\pare{\overline{X}_u~:~u\in [1-t,1-\varepsilon]} = \sigma\pare{X_u~:~u\in [\varepsilon, t]}$$
so that $\pare{\tilde{N}^{\varphi}_t}$ and all terms in \eqref{eq:renversement-2} are adapted to $\pare{{\cal F}_t^X}_{t\in [0,1]}$ the natural filtration of $X$.
\item Only $\varphi$ and its first order partial derivatives appear in equality \eqref{eq:renversement-2}. Using a density argument and a little work we may prove that the equality holds for $\varphi\in H^{1}(\R^d)$.
\item None of the quantities in the right hand side of \eqref{eq:renversement-2} involve the derivatives of the coefficient matrix $a$~:~ the dependence on the derivatives of the coefficient $a$ is totally encompassed in the logarithmic derivative of the fundamental solution $p(t,x,y)$.
\end{itemize} 

The idea is now to pick $a$ measurable in ${\cal A}\pare{\lambda, \Lambda}$ and to take a sequence of smooth $\{a_n\}_1^\infty\subset{\cal A}(\lambda, \Lambda)$ such that $a_n\longrightarrow a$ almost everywhere and to prove that there is convergence in law for the decomposition \eqref{eq:renversement-2} for any $\varphi$ belonging to a the widdest possible class of functions.
This programm has been successively performed in \cite{Rozkosz-1996} and \cite{Rozkosz-1996-2} in a more general setting of an inhomogeneous divergence operator.
\vspace{0,4 cm}

Though theoretically powerful, the Lyons-Zheng decomposition is unlikely to be directly exploitable from a numerical perspective, as we do not have access to the logarithm derivative of the transition density, even in mild cases. 
\section{Stochastic dynamics associated to transmission operators in divergence form}
\label{sec-dirichlet2}

\subsection{Skorokhod representation of the Hunt process associated to a transmission operator in divergence form}
\label{ssec-skoro}

Consider $\R^d=\bar{D}_+\cup D_-$ with $D_+$ and $D_-$ two open connected subdomains separated by a transmission boundary~$\Gamma$ that is to say$$\Gamma=\bar{D}_+\cap \bar{D}_-.$$ 

We denote 
$$D=D_+\cup D_-=\R^d\setminus\Gamma\subset\R^d.$$

For a point $x\in\Gamma$ we denote by $\nu(x)\in\R^d$ the unit normal to $\Gamma$ at point $x$, pointing to $D_+$. In the following, "$f\in C^p(\bar{D}_+;\R)\cap C^p(\bar{D}_-;\R)$" means that the restriction $f_+$ of the real valued function $f$ to $D_+$ (and the restriction $f_-$ of $f$ to $D_-$) coincides on $D_+$ (resp. $D_-$) with a function $\tilde{f}_+$ of class $C^p(\R^d)$ (resp.~$\tilde{f}_-$). Furthermore $C^p_b(E)=C^p_b(E;\R)$ will denote the set of real valued functions on $E$ of class $C^p$, bounded with bounded derivatives up to order $p$.

Assume the $a_{ij}$'s satisfy $(a_\pm)_{ij}\in C(\bar{D}_\pm ; \R)$. We may define then the co-normal vector fields ${\gamma_+}(x) := \varepsilon_+(x)\nu(x)$ and $\gamma_-(x) := \varepsilon_{-}(x)\nu(x)$, for $x\in\Gamma$.

We shall consider restricted operators and bilinear forms in the following sense. We define 
$A_+:H^1(D_+)\to H^{-1}(D_+)$ by
$$
\forall v\in H^1(D_+),\quad A_+v=\sum_{i,j=1}^dD_i\big((\varepsilon_+)_{ij}D_jv\big).$$
We define $A_-:H^1(D_-)\to H^{-1}(D_-)$ in the same manner (note that we do not specify here any domain~$\cD(A_\pm)$). Further, we define
\begin{equation*}
\cE_\pm(u,v)= \sum_{i,j=1}^d\int_{D_\pm} (a_\pm)_{ij}D_ju\,D_iv,\quad\forall u,v\in H^1(D_\pm).
\end{equation*}
We have, for $u_\pm\in H^1(D_\pm)$ with $A_\pm u_\pm \in L^2(D_\pm)$,
\begin{equation}
\label{eq:lienAEp}
\cE_\pm(u_\pm,v)=\int_{D_\pm}(-A_\pm u_\pm)v,\quad\forall v\in H^1_0(D_\pm).
\end{equation}

Imagine now that in \eqref{eq:lienAEp} we wish to take the test function in $H^1(D_\pm)$ instead of $H^1_0(D_\pm)$. There will still be a link between $A_\pm$ and $\cE_\pm$, but through Green type identities, involving 
conormal derivatives and boundary integrals.

We introduce a specific notation for the one-sided conormal derivatives on~$\Gamma$ of $u\in L^2(\R^d)$ with 
$u_\pm\in H^2(D_\pm)$. Provided the $(a_\pm)_{ij}$ are in $C^1_b(\bar{D}_\pm;\R)$ and $\Gamma$ is bounded and Lipschitz we set
\begin{equation}
\label{eq:derconorm}
\cB^{\pm}_\nu u=\nu^*{\rm Tr}^{\pm}(a_\pm\nabla u_\pm)=\sum_{i=1}^d\sum_{j=1}^d\nu_i{\rm Tr}^{\pm}\big((a_\pm)_{ij}D_ju_\pm\big)\quad\text{ on  }\;\;\Gamma
\end{equation}
where ${\rm Tr}^{\pm}: H^1(D_\pm)\to H^{1/2}(\Gamma)$ stand for the usual trace operators on $\Gamma$. 

For $g\in H^{-\frac 1 2}(\Gamma)$ and $f\in H^{\frac 1 2}(\Gamma)$ we denote by $\big(g,f\big)_\Gamma$ the action of $g$ on $f$. If both $f,g$ are in $H^{\frac 1 2}(\Gamma)$ the quantity $\big(g,f\big)_\Gamma$ coincides with the surface integral 
$\int_\Gamma gf \,d\varsigma$.

Let us recall the version of the Green identity that is used in the sequel.
\begin{proposition} [First Green identity, first version; \cite{McLean-2000}, Lemma 4.1]
\label{prop:green1}
Assume $\Gamma$ is bounded and $C^2$. Let $u\in L^2(\R^d)$ with $u_+\in H^2(D_+)$ and $u_-\in H^2(D_-)$. Assume that the coefficients $(a_\pm)_{ij}$ are in $C^1_b(\overline{D}_\pm; \R)$.
Then
$$
\cE_\pm(u_\pm,v)=\int_{D_\pm}(-A_\pm u_\pm)v\mp\Big( \cB^\pm_\nu u,{\rm Tr}^\pm(v)\Big)_\Gamma\;\;,\quad\forall v\in H^1(D_\pm).
$$
\end{proposition}

We have the following result.
\begin{theorem}
\label{Theo:stochastic-representation}
Assume $\Gamma$ is bounded and $C^2$. Assume that Assumption ${\rm ({\bf E-B})}$ is fulfilled and that for all $1\leq i,j\leq d$ $a_{ij}\in C_b^1(\overline{D}_+;\R)\cap C_b^1(\overline{D}_+;\R)$ with $a_{ij}$ possessing a possible discontinuity on $\Gamma$. 
Then, the Hunt process $\M$ associated to $({\cal E}, {\cal D}\croc{\cal E})$ is a diffusion which possesses the following Skorokhod decomposition~:~for any $k\in \{1,\dots, d\}$,
\begin{align}
\label{eq:EDS-Hunt}
X_t^k &= x_k + \int_0^t \sum_{j=1}^d \sigma_{kj}(X_s)dW_s^{j} + \int_{0}^t \sum_{j=1}^d \partial_j a_{kj}(X_s)\indi{X_s\in D}ds\nonumber\\
&\hspace{0,3 cm} + \frac{1}{2}\int_0^t \gamma_{+,k}(X_s) dK_s - \frac{1}{2}\int_0^t \gamma_{-,k}(X_s)dK_s,\hspace{0,3 cm}t\geq 0,\;{ P}^x-a.s.,\;\forall x=(x_1,\dots, x_d)\in \R^d.
\end{align}
In the above equality $\sigma:\R^d\to\R^{d\times d}$ denotes the positive square-root of coefficient $2a$ {\it i.e.} the positive matrix real valued coefficient satisfying
$$
\sigma\sigma^*(x)=2a(x),\quad\forall x\in D.$$
(Note that this coefficient exists because  $a(x)$ is non-negative definite for all $x\in D$). The process $W=(W^1,\dots,W^d)$ is a $d$-dimensional standard Brownian motion starting from zero and $(K_t)_{t\geq 0}$ denotes the unique PCAF associated to the surface measure $\varsigma(d\xi)\in S$ on $\Gamma$ through the Revuz correspondence. The process $(K_t)$ increases only at times where $X$ lies on $\Gamma$,
$$
\int_0^t \indi{X_s\in \Gamma}dK_s = K_t,\hspace{0,3 cm}t\geq 0.
$$
\end{theorem}
\begin{proof}
We apply the results of Theorem 5.5.5 in \cite{Fukushima-et-al-2011} in this context for the coordinate functions $$p_k(x_1,\dots, x_d):=x_k\;\;\;(k\in \{1,\dots, d\}).$$ and follow the ideas of \cite{Trutnau-2005} Theorem 5.2. Of course $p_k\in {\cal D}\croc{\cal E}_{b,\text{loc}}$ and $p_k$ is finely continuous on $\R^d$. Let $G$ a relatively compact open set containing $\Gamma$ and a function $f_k\in {\cal D}\croc{\cal E}_{b}$ such that $p_k=f_k$ on $G$.   
Let $\langle M^{[f_k]}\rangle$ the square bracket of $M^{[f_k]}$. Then, an easy computation from \eqref{eq:radon-measure-S} shows that the energy measure of $M^{[f_k]}$ (the Revuz measure of $\langle M^{[f_k]}\rangle$) is 
\begin{equation}
\label{eq:mufk-etc}
\mu_{\langle f_k\rangle}(dy) = \mu_{\langle M^{[f_k]}\rangle}(dy) = \langle 2a(y)\nabla f_k(y), \nabla f_k(y)\rangle\ell(dy)
\end{equation}
and we know that $\mu_{\langle f_k\rangle} = \mu_{\langle p_k\rangle}$ on $G$. It is easy to show that $\indi{G}.\mu_{\langle p_k\rangle}$ is a finite Radon measure belonging to $S_{00}$ and that $\mu_{\langle p_k\rangle}$ is a smooth measure. Then, an easy computation from \eqref{eq:revuz-correspondence-restricted} shows that 
$$
\langle M^{[f_k]}\rangle_t = \int_0^t \langle 2a(X_s)\nabla f_k(X_s), \nabla f_k(X_s)\rangle ds,\;\;k\in\{1,\dots, d\}
$$ 
and by the well-known results on stochastic representation of martingales, there exists a $d$ dimensional Brownian~motion $W=(W^1,\dots,W^d)$ such that 
$$
M^{[f_k]}_t = \int_0^t \big[\sigma(X_s)\nabla f_k(X_s)\big]^*dW_s,\;\;P^x-{\rm a.s.}\;\;\forall x\in \R^d,\;\;k\in\{1,\dots, d\}
$$
(see for e.g. \cite{Revuz-Yor-1999} Chapter V\; Theorem 3.9 and the remark following its proof).

Moreover, for any $v\in C_c^\infty(\R^d)$, using the Green Identities of Proposition \ref{prop:green1} and taking into account that $v$ is of compact support, we have~:~
\begin{align*}
{\cal E}(f_k, v) &= {\cal E}_{+}(f_{k,+},v) + {\cal E}_{-}(f_{k,-},v)\\
&=\int_{D_+}(-A_+f_{k,+})v-\Big( \cB^+_\nu f_{k,+},{\rm Tr}^+(v) \Big)_\Gamma + \int_{D_-}(-A_-f_{k,-})v + \Big( \cB^-_\nu f_{k,-},{\rm Tr}^{-}(v) \Big)_\Gamma\\
&=-\int_D \sum_{i,j=1}^d D_i \pare{a_{ij}(y)D_j f_k(y)}v(y) \indi{y\in D}\ell(dy)\\
&\hspace{3,5 cm} - \int_{\Gamma}\nu^\ast \croc{{\rm Tr}^+\pare{\varepsilon_+\nabla f_{k,+}} - {\rm Tr}^-\pare{\varepsilon_{-}\nabla f_{k,-}}}\gamma(v)d\varsigma\\
&=-\int_D \sum_{j=1}^d \partial_j a_{kj}(y)v(y) \indi{y\in D}\ell(dy) - \int_{\Gamma}\croc{{\rm Tr}^+((\varepsilon_+\nu)_k) - {\rm Tr}^-((\varepsilon_{-}\nu)_k)}vd\varsigma\\
&=-\int_D \sum_{j=1}^d \partial_j a_{kj}(y)v(y) \indi{y\in D}\ell(dy) - \int_{\Gamma}\croc{(\tilde{a}_+\nu)_k - (\tilde{a}_-\nu)_k}vd\varsigma\\
&=(\varrho_k^+, v) - (\varrho_k^-, v)
\end{align*}
with
$$
\varrho_k^\pm(dy):= -\sum_{j=1}^d \croc{\partial_j a_{kj}(y)}^\pm\indi{y\in D}\ell(dy) + \croc{(\gamma_-)_k - (\gamma_+)_k}^\pm(y)\indi{y\in \Gamma}\varsigma(dy).
$$ 
(here, the notation $[a]^+$ (resp. $[a]^{-}$) stands for the positive (resp. negative) part of some real number $a$).

Let us now proceed to show that the measures $\indi{G}.\varrho_k^\pm$ belong to $S_{00}$.

Note that $||\partial_j a_{kj}\indi{D\cap G}||_{\infty}<\infty$ and from the definition of $S_{00}$ and the Revuz correpondence \eqref{eq:revuz-correspondence-restricted}, it is not difficult to prove that the measures $-\croc{\partial_j a_{kj}}_\pm(y)\indi{y\in D}\ell(dy)$ are smooth with their corresponding additive functional writing as $\pare{-\int_0^t\croc{\partial_j a_{kj}}_\pm(X_s)\indi{X_s\in D}ds}_{t\geq 0}$.

We now turn to the surface measures $\zeta_k^\pm(dy):=\croc{(\gamma_-)_k - (\gamma_+)_k}^\pm(y)\indi{y\in \Gamma}\varsigma(dy)$. It is well-known (see e.g. \cite{Evans-Gariepy-2015} p.134 3. $(\star \star \star)$, $(\star \star \star \star)$) that there exists a universal constant $C_{0}>0$, depending only on the Lipschitz domain $D_+$, such that for all $h\in C^1(\overline{D}_+)$,
\begin{align*}
\int_{\Gamma}|h(y)|\varsigma(dy) \leq C_{0}\int_{D_+} (|\nabla h(x)| + |h(x)|)\ell(dx).
\end{align*}
Thus, for all $h\in {\cal D}\croc{{\cal E}}\cap C_c(\R^d)$, we have
\begin{align*}
\int_{\Gamma}|h(y)|\varsigma(dy) &\leq C_{0}\int_{D_+} (|\nabla h(x)| + |h(x)|)\ell(dx)\\
&\leq C_{0}(2\ell(D_+))^{1/2}\pare{\int_{\R^d}(|\nabla h(x)|^2 + |h(x)|^2)\ell(dx)}^{1/2}\\
&\leq C_{0}\sqrt{\frac{(2\ell(D_+))}{\lambda}}\pare{{\cal E}(h,h) + (h,h)}^{1/2}
\end{align*}
so that the surface measure $\varsigma(dy)$ belongs to $S_0$. Since
\begin{align*}
\forall y\in \Gamma,\;\;\;\croc{|\,(\gamma_-)_k - (\gamma_+)_k\,|(y)}^{\pm}\leq 2 |\tilde{a}_{\pm}(y)\nu(y)| \leq 2\Lambda,
\end{align*}
the surface measures $\zeta_k^\pm(dy):=\croc{(\gamma_-)_k - (\gamma_+)_k}^{\pm}(y)\indi{y\in \Gamma}\varsigma(dy)$ belong also to $S_0$.

Note that from Aronson's estimates \eqref{eq:Aronson-estimates} we retrieve the following estimations
\begin{align*}
r_1(x,y)\leq C|x-y|^{-(d-2)}\;\;\text{ if }d>2\,;\,\,r_1(x,y)\leq C\pare{\ln(1/|x-y|)\vee 1}\text{ if }d=2.
\end{align*}
Then, using the same arguments as in \cite{Fukushima-et-al-2011} (Example 5.2.2 p.255), we can assert that the measures $\zeta_k^\pm(dy)$ belong to $S_{00}$. Moreover, let $(K_t)_{t\geq 0}$ denote the PCAF associated to $\varsigma(dy)$ ; in regard of the results stated in the original article of D. Revuz (cf. \cite{Revuz-1970} p.507) we may assert that $\pare{\int_0^t\croc{(\gamma_-)_k - (\gamma_+)_k}^{\pm}(X_s)\indi{X_s\in \Gamma}dK_s}_{t\geq 0}$ is the PCAF associated to $\zeta_k^\pm(dy)$ via the Revuz correspondence.

By application of Theorem 5.5.5 in \cite{Fukushima-et-al-2011} and since all the necessary hypothesis are fulfilled, we get the decomposition \eqref{eq:EDS-Hunt} on the set $\{t\geq 0~:~t\leq \tau_{G_q}\}$ where $G_q:=\{x\in \R^d~:~|x|< q\}$. The identification of the process for all times follows by letting $q$ tend to infinity.
\end{proof}

Let $u_0\in \cD(A)$. From the Hille-Yosida theorem (\cite{brezis} Theorems VII.4 and VII.5) 
we can prove that there exists a unique function 
$$
u\in C^{1}\big([0,T];\,  L^2(\R^d) \big)\cap C\big([0,T];\, \cD(A)  \big)
$$
satisfying
\begin{equation}
\label{eq:dudtfaible}
\frac{\mathrm{d}u}{\mathrm{d}t}=Au,\quad\quad u(0)=u_0.
\end{equation}
where the first equality in \eqref{eq:dudtfaible} has to be understood in the weak sense.

Under the hypothesis of Theorem \ref{Theo:stochastic-representation}, we deduce the following Corollary.
\begin{corollary}
\label{cor:ExfX-u}
Let $0<T<\infty$. Under the conditions of  Theorem \ref{Theo:stochastic-representation}, for any $u_0\in \cD(A)$, we have 
\begin{equation}
\label{eq:ExfX-u}
E^x[u_0(X_t)]=u(t,x),\quad \forall t\in[0,T],\;\forall x\in\R^d,
\end{equation}
where $X$ is the diffusion considered in Theorem \ref{Theo:stochastic-representation} and $u$ is the solution of 
\eqref{eq:dudtfaible}.

In particular, the following transmission condition
\begin{equation}
\label{eq:transmission}\langle \varepsilon_+\nabla_x u_{+}(t,y)- \varepsilon_{-}\nabla_x u_{-}(t,y), \nu(y)\rangle = 0,\;\;\text{for\;{\it a.e.}}\;\; (t,y)\in(0,T]\times\Gamma\quad (\star)
\end{equation}
is satisfied.
\end{corollary}

\begin{proof}
In view of \eqref{eq:Tt-X} and since $\frac{d}{dt} T_tu_0=A T_t u_0$  (\cite{pazy} Thm 2.4-c)) the function $(t,x)\mapsto E^x[u_0(X_t)]$ is solution of \eqref{eq:dudtfaible}. We refer to the proofs of Proposition 3.14 and Theorem 3.1 in \cite{Etore-Martinez-2020} for the verification of the other assertions.
\end{proof}

In the light of \eqref{eq:ExfX-u} and in order to compute an approximate value of $u(t,x)$, one could think of producing a Monte Carlo method. Our article \cite{Etore-Martinez-2020} is an attempt to tackle this issue. 


\subsection{The diagonal case : link with the results of Bossy \& al. \cite{Bossy-al-2010}}
\label{sec:bossy-al}

We wish to compare the result of our Theorem \ref{Theo:stochastic-representation} with the ones in~\cite{Bossy-al-2010}. For this purpose we restrict once more the assumption on the diffusion coefficient $a(x)$ given in the setting of the preceding subsection. Namely we assume 
$$a(x)=I_d\,\,\varepsilon(x)$$
with $\varepsilon(x):=\big[  \varepsilon_+\1_{\bar{D}_+}+\varepsilon_{-} \1_{D_-} \big](x)$, $I_d$ the identity matrix and $\varepsilon_+\neq \varepsilon_{-}\in\R_+^*$. 
We are therefore considering the case of a diagonal diffusion matrix. We will also assume in the forthcoming theorems that $\Gamma$ is bounded and closed, and will consider that $D_+$ is the interior domain delimited by $\Gamma$.

We start by summing up the notions and results in \cite{Bossy-al-2010} that we need in order to do the comparison.

In the sequel we denote $(C,\cC)$ the usual canonical space, i.e. $C=C([0,\infty);\R^d)$ and $\cC=\cB(C)$ (see Pb 2.4.2 in~\cite{kara} for details).
We also denote  $(\cC_t)_{t\geq 0}$ the usual canonical filtration (see Eq. (5.3.19) in \cite{kara}).

We shall denote by $\omega$ the canonical process defined on $(C,\cC)$. Note that $\omega=(\omega_t)_{t\geq 0}$ is $(\cC_t)$-adapted.

Let us define the transmission operator ${\cal L}$ acting on functions $\varphi\in C(\R^d)\cap C^2(\bar{D}_+)\cap C^2(\bar{D}_-)$ by 
\begin{align}
\label{eq:def-cL-2}
{\cal L}\varphi(x) &= \sum_{i,j=1}^d \partial_{i}\pare{\croc{I_d\,\varepsilon}_{ij}\partial_j\varphi}(x) = \sum_{i=1}^d \varepsilon(x)\,\partial^2_{ii}\varphi(x),\;\;\forall x\in \R^d\setminus \Gamma,\;\;\nonumber\\
{\cal L}\varphi(x) &=\delta(x),\;\;\forall x\in \Gamma
\end{align}
where for any $x\in \Gamma$, $\delta(x)\in \R$ is an arbitrary value of no importance in our computations.

\begin{definition}[\cite{Bossy-al-2010}]
i ) A family of probability measures $(\P^x)_{x\in\R^d}$ on $(C,\cC)$ solves the martingale problem for the operator $\cL$ if, for all $x\in\R^d$ one has 
$\P^x(\omega(0)=x)=1$ and, for all $\varphi$ satisfying
\begin{equation}
\label{eq:dom-phi}
\varphi\in C_b(\R^d)\cap C^2_b(\bar{D}_+)\cap C^2_b(\bar{D}_-)
\end{equation}
\begin{equation}
\label{eq:trans-phi}\langle \varepsilon_+\nabla_x \varphi_{+}(y)- \varepsilon_{-}\nabla_x \varphi_{-}(y), \nu(y)\rangle = 0,\;\;\forall y\in\Gamma,\quad (\star)
\end{equation}
\end{definition}
one has that the $t$-indexed process defined by
$$
\varphi(\omega_t)-\varphi(\omega_0)-\int_0^t\cL\varphi(\omega_s)ds,\quad\forall t\geq 0
$$
is a $(\cC_t)$-martingale under $\P^x$.

ii) The martingale problem is said to be well-posed if there exists a unique family of probability measures~$(\P^x)_{x\in\R^d}$ which solves the martingale problem for the operator $\cL$.

\begin{theorem}[Theorems 2.4, 2.10 and 2.12 in \cite{Bossy-al-2010}]
\label{thm:bossy1}
For all $x\in\R^d$ consider the SDE
\begin{equation}
\label{eq:eds-bossy}
\left\{
\begin{array}{lll}
X_t&=&\ds x+\int_0^t\sqrt {2\varepsilon(X_u)}dB_u+\frac{\varepsilon_+-\varepsilon_{-}}{2\varepsilon_{-}}\int_0^t\nu(X_u)dL^0_u(Y)\\
\\
Y_t&=&\rho(X_t)\\
\end{array}
\right.
\end{equation}
where $B$ is a $d$-dimensional Brownian motion, $\rho(z)$ is the distance from $z\in\R^d$ to the boundary $\Gamma$ and $L^0_t(Y)$ stands for the (right) local time at point zero of the local martingale $Y$.

If $\Gamma$ is of class $C^3$ and compact we have:

i) There exists a weak solution to \eqref{eq:eds-bossy}. For any $x\in\R^d$ consider the law $\P^x$ of this weak solution on $(C,\cC)$. The family $(\P^x)$ is a solution to the martingale problem for $\cL$.

ii) Conversely let $(\P^x)$ be a solution to the martingale problem for $\cL$, and let $x\in\R^d$. There exists a $(\cC_t)$-Brownian motion $B$ under $\P^x$, such that $\omega$ is a weak solution to \eqref{eq:eds-bossy} driven by $B$, under~$\P^x$.
\end{theorem}

\begin{remark}
Note a change of sign in the weight in front of the local time term in \eqref{eq:eds-bossy}, compared to \cite{Bossy-al-2010}. This is because for us the exterior normal to the interior domain $D_+$ is $-\nu$.
\end{remark}

\begin{remark}
Note that in \cite{Bossy-al-2010} the authors work with right local time at point zero of $Y=\rho(X)$. Working with the symmetric local time instead, as it is often the case in the study of asymetric diffusions (e.g. \cite{Lejay-2006}) would lead to different coefficients in front of the local time term.
\end{remark}

\begin{theorem}[Theorem 2.14 in \cite{Bossy-al-2010}]
\label{thm:bossy2}
Assume $\Gamma$ is of class $C^3$ and compact. Then the martingale problem for $\cL$ is well posed and in particular there is a unique weak solution to \eqref{eq:eds-bossy} in the sense of probability law.
\end{theorem}

We then have the following main result.

\begin{theorem}
\label{thm:X-resout-pb-mart}
Assume $\Gamma$ is of class $C^2$ and compact. Consider the Hunt process $X$ ($\M$) in Theorem \ref{Theo:stochastic-representation}. For any starting point
$x\in\R^d$ consider the law $\P^x$ of $X$ on $(C,\cC)$. The family $(\P^x)$ is a solution to the martingale problem for $\cL$.
\end{theorem}

\begin{proof}
 Let $\M = \pare{\Omega, ({\cal F}_t)_{t\geq 0}, {\cal F}, (X_t)_{t\geq 0}, (P^x)_{x\in \R^d}}$  the Hunt process considered in Theorem \ref{Theo:stochastic-representation}. Let $x\in\R^d$ and let  $\varphi$ satisfying~\eqref{eq:dom-phi} and \eqref{eq:trans-phi}. We aim at applying Theorem 5.5.5. in \cite{Fukushima-et-al-2011} with $\varphi$ in order to check that
$$
\varphi(X_t)-\varphi(x)- \int_{0}^t  \cL\varphi(X_s)ds
$$
is a martingale under $P^x$ w.r.t  $(\cF_t)$.

Indeed, proceeding as in \cite{kara} p314, this implies that the law $\P^x$ induced on $(C,\cC)$ by $X$ (under $P^x$) is such that if we consider the family $(\P^x)$ this is a solution to the martingale problem for $\cL$.

\vspace{0.1cm}

The fact that $\varphi$ is finely continuous is clear. In order to check that $\varphi\in{\cal D}\croc{\cal E}_{b,\text{loc}}$ it suffices to notice that for any compact $G\subset\R^d$ the functions
$(\varphi\1_G)_+$ and $(\varphi\1_G)_+$ are respectively in $H^1(D+)$ and $H^1(D_-)$. As $\varphi\1_G$ is in $L^2(\R^d)$ and continuous across $\Gamma$, Exercise 4.5 in 
\cite{McLean-2000} implies that $\varphi\1_G\in H^1(\R^d)={\cal D}\croc{\cal E}$.

We now check that $\1_G\cdot\mu_{\langle \varphi\rangle}\in S_{00}$ for any relatively compact set $G$. We first check that $\1_G\cdot\mu_{\langle \varphi\rangle}$ is a positive and finite Radon measure. Indeed one has, as in Eq. \eqref{eq:radon-measure-S},
$$
\1_G\cdot\mu_{\langle \varphi\rangle}(\R^d)=\int_G\mu_{\langle \varphi\rangle}(dx)=2\cE(\1_G\varphi,\varphi)-\cE(\1_G\varphi^2,1)=2\cE(\1_G\varphi,\varphi)
$$
and $0\leq \cE(\1_G\varphi,\varphi)<\infty$. We now check that $||R_1\,\1_G\cdot\mu_{\langle \varphi\rangle}||_\infty<\infty$. For any $x\in\R^d$ we have, with a constants $m,M'$ depending on~$\lambda,\Lambda$,
\begin{equation}
\label{eq:maj-R1etc}
\begin{array}{lll}
|R_1\,\1_G\cdot\mu_{\langle \varphi\rangle}\,(x)|&\leq&\ds\int_0^\infty e^{-t}\Big(\int_{\R^d}p(t,x,y)|\1_G\cdot\mu_{\langle \varphi\rangle}|(dy)\Big)\,dt\\
\\
&\leq&\ds M\int_0^\infty dt\,e^{-t}\int_{\R^d}\frac{1}{t^{d/2}}e^{-\frac{|y-x|^2}{M t}}|\1_G\cdot\mu_{\langle \varphi\rangle}|(dy)\\
\\
&=&\ds  M\int_0^\infty dt\,e^{-t}\int_{G}\frac{1}{t^{d/2}}e^{-\frac{|y-x|^2}{M t}}| \langle 2a(y)\nabla\varphi(y),\nabla\varphi(y) \rangle| dy\\
\\
&=&\ds  M'\int_0^\infty dt\,e^{-t}E^x\big[ |\langle 2a(m W_t)\nabla\varphi(m W_t),\nabla\varphi(m W_t) \rangle| \1_{W_t\in G} \big]\\
\\
&\leq& \ds 2M' \Lambda\sup_{z\in G}|\nabla u(mz)|^2\int_0^\infty dt\,e^{-t}dt\;<\;\infty\\
\end{array}
\end{equation}
Here we have used \eqref{eq:Aronson-estimates} at the second line. At the third line we have used \eqref{eq:radon-measure-S} and computations similar to the ones leading to \eqref{eq:mufk-etc}. At the last line we have used {\rm ({\bf E-B})}. As the bound in \eqref{eq:maj-R1etc} does not depend on $x$ we have proven  $||R_1\,\1_G\cdot\mu_{\langle \varphi\rangle}||_\infty<\infty$.

Let $v\in C^\infty_c(\R^d)$. We have, using in particular Proposition \ref{prop:green1}, the smoothness of $\varphi$, ${\rm Tr}^+(v)={\rm Tr}(v+)={\rm Tr}^-(v)={\rm Tr}(v-)={\rm Tr}(v)$
and \eqref{eq:trans-phi},
$$
\begin{array}{lll}
\cE(\varphi,v)&=&\cE_+(\varphi,v)+\cE_-(\varphi,v)\\
\\
&=&\ds-\int_{D_+}(A_+ \varphi_+)v-\int_{D_-}(A_- \varphi_-)v+\Big( \cB^-_\nu u-\cB^+_\nu u,{\rm Tr}(v)\Big)_\Gamma\\
\\
&=&\ds-\int_D\cL\varphi\,v+\int_\Gamma \langle {\rm Tr}(\varepsilon_+\nabla_x \varphi_{+})- {\rm Tr}(\varepsilon_{-}\nabla_x \varphi_{-}), \nu\rangle\,v\,d\varsigma\\
\\
&=&\ds-\int_D\cL\varphi\,v.
\end{array}
$$
The function $\cL\varphi$ is piecewise continuous and bounded and proceeding as above one may check that the positive and negative parts $(-\cL\varphi)^\pm$ satisfy
$\1_G\cdot\big(  (-\cL\varphi)^\pm dx \big)\in S_{00}$ for any relatively compact set $G$. 

We denote $A^\pm$ the PCAF's related to Revuz measures $(-\cL\varphi)^\pm dx$. We set
$$
N^{[\varphi]}=-A^++A^-$$
and notice that following Example 5.1.1 in \cite{Fukushima-et-al-2011} we have
$N^{[\varphi]}_t=-\int_0^t(-\cL\varphi)^+(X_s)ds+\int_0^t(-\cL\varphi)^-(X_s)ds=\int_0^t\cL\varphi(X_s)ds$
(note that we use the fact that $\int_0^t\1_{X_s\in\Gamma}ds=0$ $\P^x$-a.s., therefore the arbitrary value of $\cL\varphi(z)$ for~$z\in\Gamma$ causes no issue).

We now apply Theorem 5.5.5. in \cite{Fukushima-et-al-2011}. We have
$$
 \varphi(X_t)-\varphi(x)=M^{[\varphi]}_t+\int_0^t\cL\varphi(X_s)ds$$
 with $M^{[\varphi]}$ which a martingale (as for example the $M^{[f_k]}$'s are martingales in the proof of Theorem \ref{Theo:stochastic-representation}). The proof is completed.
\end{proof}

In view of Theorems \ref{thm:bossy1}, \ref{thm:bossy2} and \ref{thm:X-resout-pb-mart} we immediately get the following corollary.

\begin{corollary}
Assume $\Gamma$ is of class $C^3$ and compact.

The solutions of \eqref{eq:EDS-Hunt} and \eqref{eq:eds-bossy}, with starting point $x\in\R^d$, have the same distribution on $(C,\cC)$.

In particular there is uniqueness in the sense of probability law of the weak solutions of \eqref{eq:EDS-Hunt}.
\end{corollary}

%
%
%
%
%

Going a bit further in the analysis we may do an identification in the strong sense of the terms appearing in \eqref{eq:EDS-Hunt} and \eqref{eq:eds-bossy}.

\begin{corollary}
Assume $\Gamma$ is of class $C^3$ and compact. 

Consider the Hunt process $\M = \pare{\Omega, ({\cal F}_t)_{t\geq 0}, {\cal F}, (X_t)_{t\geq 0}, (P^x)_{x\in \R^d}}$ in Theorem \ref{Theo:stochastic-representation}, which is such that $X$
 solves~\eqref{eq:EDS-Hunt} under $P^x$.

Then one also has
\begin{equation*}
\label{eq:eds-bossy-2}
\left\{
\begin{array}{lll}
X_t&=&\ds x+\int_0^t\sqrt {2\varepsilon(X_u)}dB_u(X)+\frac{\varepsilon_+-\varepsilon_{-}}{2\varepsilon_{-}}\int_0^t\nu(X_u)dL^0_u(Y)\\
Y_t&=&\rho(X_t)\\
\end{array}
\right.
\end{equation*}
with $B(X)$ a Brownian motion under $P^x$ (that is $X$ is also a weak solution to \eqref{eq:eds-bossy}).

Moreover one has $B(X)=W$ (with $W$ the Brownian motion driving \eqref{eq:EDS-Hunt}) and
\begin{equation}
\label{eq:identif-tps-loc}
K_t=\frac{1}{\varepsilon_{-}}L^0_t(\rho(X)),\quad\forall t\geq 0.
\end{equation}
\end{corollary}

\begin{proof}
Pick $x\in\R^d$. Let us rewrite \eqref{eq:EDS-Hunt} in the matrix form in the case of interest. One has
$$
X_t=x+\int_0^t\sqrt{2\varepsilon}(X_u)dW_u+\frac 1 2 \int_0^t(\gamma_+(X_u)-\gamma_-(X_u))dK_u
$$
under $P^x$.
We set now
$$
G(X):=X_t-x=\int_0^t\sqrt{2\varepsilon}(X_u)dW_u+\frac 1 2 \int_0^t(\gamma_+(X_u)-\gamma_-(X_u))dK_u
$$
and notice that $G(\cdot)$ is obviously measurable and that 
$$
\Phi(X):=G(X)-X_t-x=0\quad P^x-\text{a.s.}
$$
Noting that $\Phi(\cdot)$ is measurable and remembering the definition of $\P^x$ as the law of $X$ on $(C,\cC)$ under $P^x$ it is obvious that one has
$
P^x(\Phi(X)=0)=1=\P^x(\Phi(\omega)=0)
$.

Therefore under $\P^x$ one has a.s.
$$
G(\omega)=\omega_t-x
$$
but from Theorem \ref{thm:bossy1} Point ii) one has 
\begin{equation}
\label{eq:G1}
G(\omega)=\int_0^t\sqrt {2\varepsilon(\omega_u)}dB_u(\omega)+\frac{\varepsilon_+-\varepsilon_{-}}{2\varepsilon_{-}}\int_0^t\nu(\omega_u)dL^0_u(\rho(\omega))
\end{equation}
where $B(\omega)$ is a $(\cC_t)$-Brownian motion under $\P^x$. In the notation we have stressed that $B(\omega)$ is constructed from the paths of $\omega$ (through a measurable mapping).

Applying $G(\cdot)$ viewed as in \eqref{eq:G1} to $X$ we get the first part of the corollary, that is that $X$ solves \eqref{eq:eds-bossy} \eqref{eq:eds-bossy-2} driven by~$B(X)$ (to check that
$B(X)$ is a B.m. under $P^x$ we have to check that the increments of $B(X)$ are independent and that $B_t(X)-B_s(X)$ for any $s<t$ and distributed as a $\cN_d(0,(t-s)I_d)$, by identifying the law of $X$ under~$P^x$ with the one of $\omega$ under $\P^x$).

Using the uniqueness of the decomposition of a semimartingale and recalling that $\gamma_\pm=a_\pm\nu$ we have
\begin{equation}
\label{eq:indent-mart}
\int_0^t\sqrt{2\varepsilon}(X_u)dW_u=\int_0^t\sqrt {2\varepsilon(X_u)}dB_u(X)\quad\forall t\geq 0
\end{equation}
and
\begin{equation}
\label{eq:indent-VF}
\int_0^t\nu(X_s)dK_s=\int_0^t\nu(X_s)\frac{1}{\varepsilon_{-}}dL^0_s(\rho(X))\quad\forall t\geq 0.
\end{equation}
From \eqref{eq:indent-mart} and \eqref{eq:ell} one has 
$$
0\leq \langle W-B(X)\rangle_t\leq \frac{1}{2 \lambda} \int_0^t2\varepsilon(X_u)d\langle W-B(X)\rangle_u\leq \frac{1}{2 \lambda} \big\langle\int_0^\cdot\sqrt{2\varepsilon}(X_u)d (W-B(X))_u\big)\rangle_t=0
$$
where the bracket has to be understood in the multidimensional sense (matrix of brackets). Therefore $W\equiv B(X)$ using Proposition IV.1.12 in \cite{Revuz-Yor-1999}.

It remains to use the componentwise meaning of \eqref{eq:indent-VF} in order to check \eqref{eq:identif-tps-loc}. For the use of exposure we assume for a while that $d=2$, with $\nu=(\nu_1,\nu_2)^T$. Using the fact that if $\nu_1(y)=0$ then 
$\nu_2(y)\neq 0$, $y\in\Gamma$, we get for any $t\geq 0$,
$$
\begin{array}{lll}
K_t&=&\ds \int_0^t\Big( \1_{\nu_1(X_s)\neq 0}\frac{\nu_1(X_s)}{\nu_1(X_s)}+  \1_{\nu_1(X_s)= 0}\frac{\nu_2(X_s)}{\nu_2(X_s)} \Big) dK_s\\
\\
&=&\ds \int_0^t\Big( \1_{\nu_1(X_s)\neq 0}\frac{\nu_1(X_s)}{\nu_1(X_s)}+  \1_{\nu_1(X_s)= 0}\frac{\nu_2(X_s)}{\nu_2(X_s)} \Big) \frac{1}{\varepsilon_{-}}dL^0_s(\rho(X))\\
\\
&=&\ds \int_0^t\Big( \1_{\nu_1(X_s)\neq 0}+  \1_{\nu_1(X_s)= 0} \Big) \frac{1}{\varepsilon_{-}}dL^0_s(\rho(X))\\
\\
&=&\ds\frac{1}{\varepsilon_{-}}L^0_t(\rho(X))
\end{array}
$$
We claim that the above reasoning can easily be extended to $d>2$. The proof is completed.
\end{proof}





\end{document}